\documentclass[12pt,a4paper]{amsart}

\usepackage{color}
\usepackage{amsmath,amssymb,latexsym,amsthm}
\usepackage{amscd}
\usepackage[utf8]{inputenc}

\usepackage{tikz}
\usetikzlibrary{arrows}

\newtheorem{Thm}{Theorem}[section]
\newtheorem{Lemma}[Thm]{Lemma}
\newtheorem{Prop}[Thm]{Proposition}
\newtheorem{Def}[Thm]{Definition}
\newtheorem{Cor}[Thm]{Corollary}
\newtheorem{Rem}[Thm]{Remark}

\newtheorem{Example}[Thm]{Example}

\newtheorem{Quest}{Question}

\newcommand{\dd}{\delta}
\newcommand{\e}{\varepsilon}

\newcommand{\dist}{\textrm{dist}}

\newcommand{\Lc}{\textrm{L}}
\newcommand{\dns}{\textrm{DN-S}}
\newcommand{\dss}{\textrm{DSS}}

\newcommand{\ds}{\textrm{d}}
\newcommand{\dsn}{\textrm{dn}}
\newcommand{\supp}{\textrm{supp\,}}

\newcommand{\To}{\longrightarrow}

\def\PO{\operatorname{PO}}

\newcommand{\U}{\mathcal{U}}

\newcommand{\N}{\mathbb{N}}

\newcommand{\jno}{\textrm \j}

\begin{document}

\title[Disjointly non-singular operators]{Disjointly non-singular operators:\\ 
Extensions and local variations}

\thanks{Supported in part by MICINN (Spain), Grant PID2019-103961GB-C22.
}

\author[M.\ Gonz\'alez]{Manuel Gonz\'alez}
\address{Departamento de Matem\'aticas, Facultad de Ciencias, Universidad de Canta\-bria, E-39071 Santander, Spain} \email{manuel.gonzalez@unican.es}

\author[A.\ Martin\'on]{Antonio Martin\'on}
\address{Departamento de An\'alisis Matem\'atico, Facultad de Ciencias, Universidad de La Laguna, E-38271 La Laguna (Tenerife), Spain} \email{anmarce@ull.es }



\begin{abstract}
The disjointly non-singular ($\dns$) operators $T\in\Lc(E,Y)$ from a Banach lattice $E$ to a Banach space $Y$ are those operators which are strictly singular in no closed subspace generated by a disjoint sequence of non-zero vectors. When $E$ is order continuous with a weak unit, $E$ can be represented as a dense ideal in some $L_1(\mu)$ space, and we  show that each of $T\in\dns(E,Y)$ admits an extension $\overline{T}\in \dns(L_1(\mu),\PO)$ from which we derive that both $T$ and $T^{**}$ are tauberian operators and that the operator $T^{co}: E^{**}/E\to  Y^{**}/Y$ induced by $T^{**}$ is an (into) isomorphism.  Also, using a local variation of the notion of $\dns$ operator, we show that the ultrapowers of $T\in\dns(E,Y)$ are also $\dns$ operators. Moreover, when $E$ contains no copies of $c_0$ and admits a weak unit, we  show that $T\in\dns(E,Y)$ implies $T^{**}\in \dns(E^{**},Y^{**})$.
\end{abstract}

\maketitle

\thispagestyle{empty}

\section{Introduction}

In a Banach lattice $E$ we can consider two kinds of closed subspaces: those generated by a disjoint sequence of non-zero vectors, and those that are at a positive distance of every normalized disjoint sequence. The later ones are called \emph{dispersed subspaces} in  \cite{GMM:20}.  
In the study of operators acting on $E$ it is useful to consider their action on these kinds of subspaces (see \cite{FOPT:22}). 
The \emph{disjointly strictly singular} operators ($\dss$ operators, for short) were introduced in \cite{Hernandez-Salinas:89} as those operators $T:E\to Y$ from a Banach lattice $E$ into a Banach space $Y$ such that $T$ is an isomorphism on no closed subspace of $E$ generated by a disjoint sequence of non-zero vectors. These operators have been applied to  the study of the structure of Banach lattices (see \cite{FLT:16} and references therein). More recently, the disjointly non-singular operators ($\dns$ operators, for short) where introduced in \cite{GMM:20} as those operators $T:E\to Y$ from a Banach lattice $E$ to a Banach space $Y$ that are strictly singular in no closed subspace of $E$ generated by a disjoint sequence of non-zero vectors. 
The $\dns$ operators have also been  studied in \cite{Bilokopytov:21} and \cite{GM:22}. 
Note that the kernel of a $\dns$  operator is a dispersed subspace. 

By \cite[Theorem 2]{GM:97}, an operator $T: L_1\to Y$ is $\dns$ if and only it is tauberian in the sense of Kalton and Wilansky \cite{KaltonWil:76}. In this case the second conjugate $T^{**}: L_1^{**}\to Y^{**}$ and the ultrapowers $T_\U: (L_1)_\U\to Y_\U$ are also $\dns$, and the operator $T^{co}: L_1^{**}/L_1\to Y^{**}/Y$ induced by $T^{**}$ is an (into) isomorphism; see \cite{GM:97,GM:10}.

In this paper we extend these  results for $E=L_1$ to the operators in $\dns(E,Y)$ when $E$ is order continuous with a weak unit. Our main tool is the fact that in this case $E$ admits a representation and a dense sublattice of some $L_1(\mu)$ space with $\mu$ a probability measure. 
We characterize the operators in $\dns(E,Y)$ in terms of their action over the normalized sequences $(x_n)$ in $E$ satisfying $\lim_{n\to\infty} \mu(\supp x_n) =0$. As a consequence, $T$ is an isomorphism on the closed band $E_A$ of $E$ generated by a measurable set $A$ when $\mu(A)$ is small enough. 
Moreover, using the push-out construction, we show that every operator $T\in\dns(E,Y)$ admits an extension $\overline{T}\in \dns(L_1(\mu),\PO)$, where $\PO$ is the push-out Banach space. From this result, we derive that each $T\in\dns(E,Y)$ is a tauberian operator such that $T^{**}$ is tauberian and $T^{co}$ is an (into) isomorphism. 
Also, using a local variation of the notion of $\dns$ operator, we  prove that the class of $\dns$ operators is preserved by ultrapowers, we give an example showing that it is not preserved by ultraproducts, and we introduce  and study the $n,r$-dispersed subpaces, a local variation of the notion of dispersed subspace. Moreover, when $E$ contains no copies of $c_0$ and admits a weak unit, we  show that $T\in\dns(E,Y)$ implies $T^{**}\in \dns$. 

\subsection*{Notation}
Throughout the paper $X$ and $Y$ are Banach spaces, $E$ is a Banach lattice and $E_+=\{ x\in E : x\geq 0\}$. The unit sphere of $X$ is  $S_X=\{x\in X : \|x\|=1\}$, and for a sequence $(x_n)$ in $X$, $[x_n]$ denotes the closed subspace generated by $(x_n)$.
We also  denote $\ds(E)=\{(x_n)\subset E\setminus \{0\} : (x_n) \textrm{ disjoint} \}$, and $\dsn(E)=\{(x_n) \subset S_E : (x_n) \textrm{ disjoint} \}$. 

Operators always are linear and continuous, and $\Lc(X,Y)$ denotes the set of all operators from $X$ into $Y$. Given $T\in \Lc(X,Y)$, $N(T)$ is the kernel of $T$, $R(T)$ is the range of $T$, and we denote by $T_M$ the restriction of $T\in \Lc(X,Y)$ to a closed subspace $M$ of $X$. 

An operator $T\in\Lc(X,Y)$ is {\it strictly singular} if there is no closed infinite dimensional subspace $M$ of $X$ such that $T_M$ is an isomorphism; the operator $T$ is {\it upper semi-Fredholm} if $N(T)$ is finite dimensional and $R(T)$ is closed; and $T$ is \emph{tauberian} if its second conjugate $T^{**}:X^{**}\to Y^{**}$ satisfies $T^{**-1}(Y)=X$ \cite{KaltonWil:76}; equivalently, if the operator $T^{co}: X^{**}/X \to Y^{**}/Y$ induced by $T^{**}$ is injective. 
We refer to \cite{GSK:95} for the properties of $T^{co}$.

\section{Preliminaries}

An operator $T\in\Lc(E,Y)$ is \emph{disjointly strictly singular}, and we write $T\in \dss (E,Y)$, if there is no $(x_n)\in \ds(E)$ such that $T_{[x_n]}$ is an isomorphism. The class $\dss$ was introduced by Hern\'andez and Rodr\'\i guez-Salinas in \cite{Hernandez-Salinas:89} and \cite{Hernandez:90}. 

An operator $T\in \Lc(E,Y)$ is {\it disjointly non-singular}, and we write $T\in \dns (E,Y)$,  if there is no $(x_n)\in \ds(E)$ such that $T_{[x_n]}$ is strictly singular.  
These operators were introduced in \cite{GMM:20}, and studied in \cite{Bilokopytov:21} and \cite{GM:22}. Note that $\dns(L_1,Y)$ is the set of tauberian operators from $L_1$ into $Y$ (see \cite{GM:97,GMM:20}). We refer to \cite{GO:90} and \cite{GM:10} for information on tauberian operators. 
A closed subspace $M$ of $E$ is \emph{dispersed} if there is no $(x_n)\in \dsn(E)$ such that $\lim_{n\to\infty} \dist(x_n,M)=0$. 

A sequence $(x_n)$ in $E$ is \emph{unbounded norm convergent (or un-convergent)} to $x\in E$  if $(|x_n-x|\wedge u)$ converges in norm to $0$ for each $u\in E_+$ \cite{Troitsky:04}. In this case we write  $x_n  \overset{un}{\To} x$. 

The disjointly non-singular operators can be characterized as follows.

\begin{Thm}\label{dns-char} \cite[Theorems 2.8 and 2.10]{GMM:20}
For an operator $T\in\Lc(E,Y)$, the following assertions are equivalent:
\begin{enumerate}
\item $T$ is disjointly non-singular.
\item For every $(x_n)\in \ds(E)$, the restriction $T_{[x_n]}$ is an upper semi-Fredholm operator.
\item For every $(x_n)\in \dsn(E)$, $\liminf_{n\to \infty} \|Tx_n\|>0$. 
\item For every compact operator $S\in\Lc(E,Y)$, $N(T+S)$ is dispersed.
\end{enumerate}
\end{Thm}


\begin{Thm}\label{dns-char-oc} \cite[Theorem 5.3]{Bilokopytov:21}
Suppose that $E$ is order continuous. For $T\in\Lc(E,Y)$, the following assertions are equivalent:
\begin{enumerate}
\item $T$ is disjointly non-singular.
\item For no normalized $un$-null sequence $(x_n)$ we have $\lim_{n\to\infty} \|Tx_n\|=0$.
\item There exists $r>0$ such that for every $(x_n)\in \dsn(E)$, $\liminf_{n\to \infty} \|Tx_n\|>r$.
\end{enumerate}
\end{Thm}

A Banach lattice $E$ is \emph{order continuous} if every net in $E$ decreasing in order to $0$ converges in norm to $0$; and  
a \emph{weak unit} in $E$ is a vector $e\in E_+$ such that $\||x| \wedge e\|=0$ implies $x=0$. We refer to \cite{LT-II:79,Wnuk:99} for information on order continuous Banach lattices. 

\subsection*{Representation of order continuous Banach lattices}\label{rep-BL} In \cite[Theorem 1.b.14]{LT-II:79} it is shown that every order continuous Banach lattice $E$ with a weak unit admits a representation as a K\"othe function space, in the sense that there exists a probability space $(\Omega,\Sigma,\mu)$ so that
\begin{itemize}
\item $L_\infty(\mu) \subset E \subset L_1(\mu)$ with $E$ dense in $L_1(\mu)$ and $L_\infty(\mu)$ dense in $E$, 
\item $\|f\|_1\leq \|f\|_E\leq 2\|f\|_\infty$ when $f\in L_\infty(\mu)$, 
\item the order in $E$ coincides with the one induced by $L_1(\mu)$. 
\end{itemize}

In the paper we will use this representation without further comments. For vectors in $L_1(\mu)$, we denote by $x_n \overset{\mu}{\To} x$ the \emph{convergence in measure.} 

Among the order continuous Banach lattices with a weak unit we have some rearrangement  invariant (r.i., for short) function spaces on $(0,1)$. Besides $L_p(0,1)$  ($1\leq p<\infty$), the most commonly used r.i. function spaces on $(0,1)$ are the Orlicz spaces and the Lorentz spaces (see \cite[Section 2a]{LT-II:79}). Below we give a brief description of the second ones.   


\begin{Example}  
Let $1\leq p<\infty$ and let $W$ be a positive non-increasing continuous function on $(0,1]$ so that $\lim_{t\to 0}W(t)=\infty$, $W(1)=1$ and  $\int_0^1 W(t)=1$. The \emph{Lorentz function space $L_{W,p}(0,1)$} is the space of all measurable functions $f$ on $(0,1)$ such that 
$$
\|f\|_{W,p} =\left(\int_0^1 f^*(t)^p W(t)dt \right)^{1/p}<\infty, 
$$
where $f^*$ is the decreasing rearrangement of $|f|$. 

The space $L_{W,p}(0,1)$ is a r.i.\ function space on $(0,1)$ different from $L_1(0, 1)$. 
\end{Example}

%
%
%
%

The following result will be useful. 

\begin{Lemma}\emph{\cite[Corollary 2.12, Theorem 4.6]{DBT:17}}\label{unconv-e} 
Let $E$ be an order continuous Banach lattice with a weak unit $e$, and let $(x_n)\subset E$. Then the following statements are equivalent: 
\begin{enumerate}
\item $x_n \overset{un}{\To} 0$. 
\item $(|x_n|\wedge e)$ converges in norm to $0$. 
\item $x_n \overset{\mu}{\To} 0$. 
\end{enumerate}
\end{Lemma}

For an order continuous Banach lattice $E$ with a weak unit $e$, we define the \emph{support} of $x\in E$ as $\supp x =\{t\in \Omega : x(t)\neq 0\}$. 

\begin{Cor} \label{disj=un-null} 
Suppose that $E$ is order continuous with a weak unit. Then each  sequence $(x_k)$ in $E$ with $\lim_{n\to \infty} \mu \left(\supp x_n\right) =0$ is un-convergent to $0$.  
\end{Cor}
\begin{proof} 
Note that $\lim_{n\to \infty} \mu \left(\supp x_n\right) =0$ implies $x_n \overset{\mu}{\To} 0$. 
\end{proof}


\subsection*{Ultraproducts of spaces and operators}\label{ultra-sect}

Let $I$ be a set admitting a non-trivial ultrafilter $\U$ and let  $(X_i)_{i\in I}$ and $(Y_i)_{i\in I}$ be families of Banach spaces. The ultraproduct  $(X_i)_\U$ of $(X_i)_{i\in I}$ is defined as the quotient of $\ell_\infty(I,X_i)$ by the closed subspace 
$$N_\U(X_i)=\{(x_i)\in \ell_\infty(I,X_i) :\lim_{i\to\U} \|x_i\|=0\}.$$ 
The element of $(X_i)_\U$ which has $(x_i)\in\ell_\infty(I,X_i)$ as a representative is denoted $[(x_i)]$. 

When $X_i=X$ for each $i\in I$, we denote the ultraproduct by  $X_\U$, and we call it  an ultrapower of $X$. 

If each $X_i$ is a Banach lattice then  $(X_i)_\U$ has a natural structure of Banach lattice: $[(x_i)] \leq [(y_i)]$ if there exists $(z_i) \in N_\U(X_i)$ such that $x_i+z_i \leq y_i$ for each $i\in I$. 

If $(T_i)_{i\in I}$ is a bounded family of operators with $T_i\in \Lc(X_i,Y_i)$ for each $i\in I$, the ultraproduct $(T_i)_\U\in \Lc((X_i)_\U, (Y_i)_\U)$ is defined by $(T_i)_\U[(x_i)]= [(T_ix_i)]$. When $T_i=T$ for each $i\in I$, we write $T_\U$ which is called an ultrapower of $T$. We refer to \cite[Chapter 8]{DJT:95} or \cite{Heinrich} for additional information on ultraproducts of spaces and operators. 


\section{Disjointly non-singular operators}

We begin with a complement to 
Theorem \ref{dns-char-oc}. 

\begin{Thm}\label{supp-dns}
Let $E$ be an order continuous Banach lattice with a weak unit. For $T\in\Lc(E,Y)$, the following assertions are equivalent: 
\begin{enumerate}
\item $T$ is disjointly non-singular. 
\item There exists $r>0$ such that for every  $(x_n)$ in $S_E$ with $\lim_{n\to \infty} \mu\left(\supp x_n \right) =0$,  $\liminf_{n\rightarrow \infty} \|Tx_n\|>r$. 
\item For every  $(x_n)$ in $S_E$ with $\lim_{n\to \infty} \mu\left(\supp x_n \right) =0$,  $\liminf_{n\rightarrow \infty} \|Tx_n\|>0$. 
\item There is $r>0$ such that for every  $x\in S_E$ with $\mu(\supp x)<r$ we have  $\|Tx\|>r$.
\end{enumerate}
\end{Thm}
\begin{proof}
(1)$\Rightarrow$(2) Suppose that $T$ is disjointly non-singular. Without loss of generality, we can assume that  $\|T\|=1$. By Theorem \ref{dns-char-oc}, there is $r>0$ such that for every disjoint sequence $(z_n)$ in $S_E$, $\liminf_{n\to \infty} \|Tz_n\|>r$. 

If (2) fails, then we can find a  sequence $(x_n)$ in $S_E$ with $\lim_{n\to \infty} \mu\left(\supp x_n \right) =0$ and  $\liminf_{n\to \infty} \|Tx_n\|<r/2$. Passing to a subsequence if necessary, we can assume that $\limsup_{n\to \infty} \|Tx_n\|<r/2$ and $\sum_{n=1}^\infty \mu\left(\supp x_n \right) <\infty$. We denote $A_n= \cup_{k=n}^\infty \supp x_n$ and $B_n=\Omega \setminus A_n$. 
Then $\mu(A_n)\to 0$, $(B_n)$ increases to $\Omega$ and, since $E$ is order continuous, $(x\chi_{B_n})$ converges in norm to $x$ for every $x\in E$ \cite[Theorem 1.1]{Wnuk:99}. 

First, we choose $n_1>1$ such that $\|x_1- x_1 \chi_{B_{n_1}}\|<1/2$, and denote $y_1= x_1 \chi_{B_{n_1}}$. Note that $\|Tx_1- Ty_1 \|<1/2$ and $|y_1|\wedge |x_j|=0$ for $j\geq n_1$. 

Next, we choose $n_2>n_1$ such that $\|x_{n_1}- x_{n_1} \chi_{B_{n_2}}\|<1/3$, and denote $y_2= x_{n_1} \chi_{B_{n_2}}$. Note that $\|Tx_{n_1}- Ty_2 \|<1/3$ and $|y_i|\wedge |x_j|=0$ for $i=1,2$ and $j\geq n_2$. 

Continuing in this way we obtain a disjoint sequence $(y_n)$ such that $\|y_n\|\to 1$ as $n\to\infty$ and $\limsup_{n\to \infty} \|Ty_n\|<r/2$. Thus taking $z_n=y_n/\|y_n\|$, we obtain a normalized disjoint sequence $(z_n)$ with $\limsup_{n\to \infty} \|Tz_n\|<r/2$, and we get a contradiction.

(2)$\Rightarrow$(3) is trivial. 

(3)$\Rightarrow$(4) If (4) fails, we can find a sequence $(x_n)$ in $S_E$ with $\mu(\supp x_n)<1/n$ and $\|Tx_n\|<1/n$. So (3) also fails. 

(4)$\Rightarrow$(1) For every disjoint sequence $(x_n)$ in $S_E$, $\liminf_{n\to \infty} \|Tx_n\|\geq r$. Thus Theorem \ref{dns-char-oc} implies that $T$ is disjointly non-singular. 
\end{proof}

If $E$ is order continuous with a weak unit, for each measurable set $A$ with $\mu(A)>0$, the set $E_A=\{x\in E : \supp x\subset A\}$ is a closed band in $E$. 

From part (4) of Theorem  \ref{supp-dns} we derive:

\begin{Cor}\label{cor1}
If $E$ is order continuous with a weak unit and $T\in\dns(E,Y)$, then there exists $r>0$ such that, when $A$ is a measurable set with $0<\mu(A)<r$, the restriction of $T$ to $E_A$ is an isomorphism. 
\end{Cor}

As a consequence we get: 

\begin{Cor}
Let $E$ be an order continuous  r.i. function space on $[0,1]$ and suppose that $\dns(E,Y)\neq \emptyset$. Then $Y$ contains a subspace isomorphic to $E$. 
\end{Cor}
\begin{proof} 
Note that the characteristic function $\chi_{[0,1]}$ is a weak unit in $E$. Moreover, as in \cite[Section 2.b]{LT-II:79}, for $0<s<\infty$ we consider the linear map $D_s$ defined on the space of measurable functions on $[0,1]$ by 
\begin{displaymath}
(D_sf)(t)=\left\{\begin{array}{cc}
   f(t/s),  &t\leq\min\{1,s\}  \\
   0,  &s<t\leq 1\quad \textrm{(in case $s<1$).} 
\end{array}\right. 
\end{displaymath}

 Clearly  $D_s$ has norm one on $L_\infty[0,1]$ and norm $s$ on $L_1[0,1]$. Thus, a result of  Calder\'on (see  \cite[Theorem 2.a.10]{LT-II:79}) implies that $D_s$ is bounded on $E$ with norm $\leq \max\{1,s\}$. 
 
If $1<s<\infty$ and $r=s^{-1}$, then  $D_r$ is injective and $D_rD_sf= \chi_{[0,r]}f$ (see \cite[Section 2.b]{LT-II:79}); hence $D_r$ is an isomorphism of $E$ onto $E_{[0,r]}$. 
Moreover, if $T\in \dns(E,Y)$ then Corollary \ref{cor1} implies that  $T$ is an isomorphism on $E_{[0,r]}$ for $r$ small enough. 
\end{proof}


\section{Push-outs and $\dns$ operators}

Suppose that $E$ is order continuous with a weak unit. We denote by $\jno:E\To L_1(\mu)$ the inclusion of $E$ into $L_1(\mu)$, which is a (continuous) operator.

Given an operator $T:E\To Y$, the \emph{push-out diagram} for the pair $\jno,T$ is  
$$\begin{CD}
E @>T>> Y\\
@V{\jno}VV @V{\overline{\jno}}VV\\
L_1(\mu) @>{\overline{T}}>> \PO 
\end{CD}$$
where $\Delta=\{(Tx,-\jno x) : x\in E\}$ is a subspace of $Y\oplus_1 L_1(\mu)$ with closure $\overline{\Delta}$,  $PO$ is the quotient $\left(Y\oplus_1 L_1(\mu)\right)/ \overline{\Delta}$, and the operators $\overline{\jno}$ and $\overline{T}$ are defined by  $\overline{\jno} y= (y,0)+ \overline{\Delta}$ and $\overline{T}f= (0,f)+ \overline{\Delta}$. See \cite[Section 1.3]{CG:97}. 

Note that $\overline{\jno}$ and $\overline{T}$ are continuous  because they are restrictions of the quotient map onto $\PO$, and the push-out diagram is commutative: $\overline{\jno}T= \overline{T} \jno$. 

\begin{Prop}\label{DSS-j} 
Let $E$ be a Banach lattice.
\begin{enumerate}
\item The inclusion $\jno:E\to L_1(\mu)$ is disjointly strictly singular ($\jno \in \dss$) if and only if for every $(x_n)\in \dsn(E)$, $\|\jno x_n\|_1 \to 0$ as $n\to\infty$. 
\item If $E$ is an r.i.\ function space on $(0,1)$  different from $L_1(0, 1)$ then $\jno:E\to L_1(0, 1)$ is always $\dss$. 
\end{enumerate}
\end{Prop}
\begin{proof}  
(1) For the direct implication, suppose that $(x_n)\in \dsn(E)$ and $C=\inf_n \|\jno  x_n\|_1 >0$. Note that $(x_n)$ is an unconditional basic sequence and  

\begin{center}
$\left\|\jno\left(\sum_{i=1}^\infty a_ix_i \right)\right\|_1=  \left\|\sum_{i=1}^\infty a_i\jno x_i\right\|_1= \sum_{i=1}^\infty |a_i|\cdot\|\jno x_i\|_1$
\end{center}
because $(\jno x_i)$ is a disjoint sequence in $L_1(\mu)$. Therefore, 

\begin{center}
$\left\|\jno\left(\sum_{i=1}^\infty a_ix_i \right)\right\|_1\geq  C\sum_{i=1}^\infty |a_i|\geq  C\left\|\sum_{i=1}^\infty a_ix_i \right\|,$
\end{center}
and $\jno$ is an isomorphism on $[x_n]$. 

The converse implication is immediate. 

(2) is proved in 
\cite[Corollary 4.4]{GHSS:00}. 
\end{proof}


The following result can be found in  \cite{FHKT:16}, Proposition 1.1 and post comment. 

\begin{Lemma}\label{FHKT} 
Let $E$ be an order continuous with a weak unit. 
\begin{enumerate}
\item For every closed subspace $M$ of $E$, the restriction of $\jno$ to $M$ is an isomorphism, or $M$ is not dispersed. 
\item For every sequence $(x_n)$ in $S_E$, $\inf\|x_n\|_{L_1}>0$ or there exists a subsequence $(x_{n_k})$ and a disjoint sequence $(z_k)$ in $E$ such that $\|x_{n_k}-z_k\|\to 0$ as $n\to \infty$.  
\end{enumerate}
\end{Lemma} 

By Proposition \ref{DSS-j}, both  alternatives in Lemma \ref{FHKT} become dichotomies if and only if  $\jno:E\to L_1(\mu)$ is $\dss$.

We consider the injective operator $D:E\to Y \oplus_1 L_1(\mu)$ defined by $Dx= (Tx,-\jno x)$. 

%

\begin{Prop}\label{Delta}
Suppose that $E$ is an order continuous Banach lattice with a weak unit, and let $T\in\Lc(E,Y)$. 
\begin{enumerate}
\item If $T\in\dns$, then $\Delta$ is a closed subspace of $Y\oplus_1 L_1(\mu)$ and $\overline{\jno}$ is injective. 
\item If $\jno:E\to L_1(\mu)$ is $\dss$ and $\Delta$ is closed in  $Y\oplus_1 L_1(\mu)$ then $T\in\dns$. 
\end{enumerate}
\end{Prop}
\begin{proof}
(1) Suppose that $\Delta= R(D)$ is not closed. Then there is a sequence $(x_n)$ in $S_E$ such that $\|Dx_n\|= \|Tx_n\|+ \|x_n\|_1\to 0$ as $n\to\infty$. By part (2) in Lemma \ref{FHKT}, there is a subsequence $(x_{n_k})$ and a disjoint sequence $(z_k)$ in $E$ such that $\|x_{n_k}-z_k\|\to 0$. Hence $\|Tz_k\|\to 0$, and Theorem \ref{dns-char} implies that $T\notin\dns$. 

Also, $\overline{\jno} y=0$ implies $(y,0)\in \overline{\Delta}=\Delta$; thus 
$(y,0)=(Tx,-\jno x)$ for some $x\in E$. Since $\jno$ is injective,  $x=0$ and $y=Tx=0$. Hence $\overline{\jno}$ is injective. 
\medskip

(2) Suppose that $T\not\in\dns$. Then there exists $(x_n)\in\dsn(E)$ such that $\|Tx_n\|\to 0$. Since $\|\jno x_n\|_1\to 0$, $R(D)=\Delta$ is non-closed.
\end{proof}

When both $\jno$ and $\overline{\jno}$ are injective, we can see $\overline{T}$ as an extension of $T$. 

\begin{Thm}\label{th:ext}
Suppose that $E$ is an order continuous Banach lattice with a weak unit, and let $T\in\dns(E,Y)$. 
\begin{enumerate} 
\item $\overline{T} \in \dns(L_1(\mu), \PO)$; equivalently, 
$\overline{T}$ is tauberian. 
\item $T$ is a tauberian operator. 
\item $T^{**}$ is tauberian and $T^{co}$ is an (into) isomorphism. 
\end{enumerate}
\end{Thm}
\begin{proof}
(1) Let $(f_n)$ be a disjoint sequence in $S_{L_1(\mu)}$. Then $\overline{T}f_n =(0,f_n)+\Delta$ and 
$$
\|\overline{T}f_n\|= \inf_{x\in E} \|Tx\|+\|f_n-\jno x\|_{L_1}. 
$$
Since $\liminf_{n\to\infty}\|f_n- \jno x \|_{L_1}\geq 1$ for each $x\in E$, we get $\liminf_{n\to\infty} \|\overline{T} f_n\|_{L_1}\geq 1$, hence $\overline{T} \in \dns$ by Theorem \ref{dns-char}. 
\medskip

(2) Note that $\overline{T} f=0$ iff $(0,f)\in\Delta$ iff $f=\jno x$ for some $x\in E$ and $Tx=0$. Hence $N(\overline{T})= \jno\left(N(T)\right)$ and $\jno$ is an isomorphism on $N(T)$. Since $\overline{T}$ is tauberian, $N(\overline{T})$ is reflexive, hence so is $N(T)$. 

Now, if $T\in \dns$ and $S\in\Lc(E,F)$ is compact, then $T+S\in\dns$ \cite[Corollary]{GMM:20}. Therefore  $N(T+S)$ is reflexive for each compact $S$, hence $T$ is tauberian by the main result of \cite{GO:90}. 
\medskip

(3) The argument we gave in the proof of (2) shows that each $T\in \dns(E,Y)$ is supertauberian in the sense of \cite{GM:95}, because each reflexive subspace of $L_1(\mu)$ is superreflexive and supertauberian operators admit a perturbative characterization: $T\in\Lc(X,Y)$ is supertauberian if and only if $N(T+K)$ is superreflexive for each compact operator $K\in\Lc(X,Y)$ \cite[Theorem 15]{GM:95}. Moreover, if $T$ is supertauberian then  $T^{co}$ is an (into) isomorphism \cite[Proposition 6.5.3]{GM:10}, and the last fact implies $T^{**}$ tauberian because $(T^{co})^{**}\equiv (T^{**})^{co}$ is injective in this case. 
\end{proof}

\begin{Quest}\label{Q1} 
Suppose that $E$ is an order continuous Banach lattice with a weak unit. 

Is it true that $\overline{T} \in \dns(L_1(\mu), \PO)$ implies  $T\in\dns(E,Y)$? 
\end{Quest}

From Theorem \ref{th:ext} we derive the following result. We observe that, for $E$, containing no copies of $c_0$ is slightly stronger than being order continuous \cite[Chapter 7]{Wnuk:99}. 

\begin{Prop}
Suppose that $E$ is a Banach lattice with a weak unit that contains no copies of $c_0$, and let $T\in\dns(E,Y)$. Then $T^{**}\in\dns(E^{**},Y^{**})$. 
\end{Prop}
\begin{proof}
Since $E$ contains no copies of $c_0$, the canonical copy of $E$ in $E^{**}$ is a projection band \cite[Theorem 1.c.4]{LT-II:79}. Thus, denoting $E^\perp =\{z\in E^{**} : |x|\wedge|z|=0\;\; \textrm{for each $x\in E$}\}$, we have that  $E^{**}=E\oplus E^\perp$. Let $P$ denote the projection on $E^{**}$ onto $E$ with kernel $E^\perp$, and let $q:Y^{**}\to Y^{**}/Y$ denote the quotient map. 

By part (3) in Theorem \ref{th:ext}, $T^{co}$ is an isomorphism (into); hence so is $qT^{**}$ on $E^\perp$. 
Therefore, given a normalized disjoint sequence $(z_n)$ in $E^{**}$ and denoting $x_n=Pz_n$ and $y_n=(I-P)z_n$, the sequence $(x_n)$ is disjoint in $E$ and there exists $C>0$ such that $\|T^{**} z_n\|\geq C\max\{\|Tx_n\|, \|qT^{**}y_n\}$ for each $n$. Hence $\liminf \|T^{**} x_n\|>0$, and we conclude $T^{**}\in\dns$. 
\end{proof}



\section{Ultraproducts of operators}

Here we prove the stability of the class of $\dns$ operators under ultrapowers when $E$ is order continuous with a weak unit. The following local variation of the notion of $\dns$ operator will be useful. 

\begin{Def}
Let $n\in\N$ and $r>0$. An operator $T\in \Lc(E,Y)$ is in the class $\dns_{n,r}$ if for each normalized disjoint $(x_i)_{i=1}^n$ in $E$ we have  $\max_{1\leq i\leq n}\|Tx_i\|\geq r$.
\end{Def}

Next result was proved in \cite{JNST:15} for operators acting on a $L_1$ space using Kakutani's representation theorem. Our proof here uses the properties of un-convergence.  

\begin{Prop}\label{dnslocal-prod}
Suppose that $E$ is order continuous with a weak unit. An operator $T\in \Lc(E,Y)$ is in $\dns$ if and only if $T\in \dns_{n,r}$ for some $n\in\N$ and $r>0$. 
\end{Prop}
\begin{proof}
If $T\in \dns_{n,r}$, then for every $(x_n)\in \dsn(E)$, $\liminf_{n\rightarrow \infty} \|Tx_n\|\geq r$. Thus, by Theorem \ref{dns-char}, $T\in\dns$. 

Conversely, suppose that  $T\in \dns_{n,r}$ for no pair $n\in\N$ and $r>0$. Then for each $n\in\N$ we can find a normalized disjoint $(x^n_i)_{i=1}^n$ with $\max_{1\leq i\leq n}\|Tx^n_i\|\leq 1/n$. 

For each $n\in\N$ we select $i(n)\in \{1,\ldots,n\}$ so that $\mu\left(\supp x^n_{i(n)}\right)\leq 1/n$, and we denote $y_n=x^n_{i(n)}$. By Corollary  \ref{disj=un-null}, $y_n \overset{un}{\To} 0$. Since $\|Ty_n\|\to 0$, Theorem \ref{dns-char-oc} implies that $T\notin\dns$. 
\end{proof}

Next we state a characterization of the class $\dns_{n,r}$ that was given in the proof of \cite[Lemma 2.2]{JNST:15} for $E=Y$ a $L_1$ space. Note that for $0\neq x \in E\subset L_1(\mu)$, we have $x/|x|\in L_\infty(\mu)$.

\begin{Prop}\label{dnslocal-char}
Suppose that $E$ is order continuous with a weak unit. Then  $T\in \Lc(E,Y)$ is in  $\dns_{n,r}$ if and only if for every $\e>0$ there is $\dd>0$ such that if $x_1,\ldots,x_n\in S_E$ and $\||x_i|\wedge |x_j|\|<\dd$ for $1\leq i<j \leq n$ then $\max_{1\leq i\leq n}\|Tx_i\|> r-\e$.
\end{Prop}
\begin{proof}
For the direct implication, if $x_1, \ldots, x_n\in S_E$ satisfy $\||x_i|\wedge |x_j|\|<\dd$ for $i\neq j$  and we define 
$$
z_i= \big(|x_i|- (|x_i|\wedge (\vee_{j\neq i} |x_j|)\big) \frac{x_i}{|x_i|},
$$
then the vectors $z_i$ are disjoint and  $1-n\dd\leq \|z_i\|\leq 1$. Applying the $\dns_{n,r}$ condition to  $(z_i/\|z_i\|)_{i=1}^n$ we get  $\max_{1\leq i\leq n}\|Tx_i\|> r-\e$ if $\delta=\delta(\e,n,\|T\|)$ is small enough.

The converse implication is immediate. 
\end{proof}

As a consequence, $\dns_{n,r}$ is stable under ultraproducts: 

\begin{Prop}\label{ultra-dns} Suppose that $E_i$ is order continuous with a weak unit for each $i\in I$. Let $\U$ be a non-trivial ultrafilter on $I$. If $(T_i)_{i\in I}$ is a bounded family with $T_i\in \dns_{n,r}(E_i,Y_i)$ for each $i\in I$ then $(T_i)_\U\in \dns_{n,r}$.
\end{Prop}
\begin{proof}
Two vectors $[(x_i)], [(z_i)]$ in  $(E_i)_\U$ are disjoint if and only if $\lim_{i\to\U} \||x_i|\wedge|y_i|\|=0$. In this case, for each $\dd>0$ we can choose the representatives $(x_i), (y_i)$ so that $\||x_i|\wedge|y_i|\|< \dd$ for every $i\in I$. Since $(T_i)_{i\in I}$ is bounded, for each $\e>0$ we can choose $\delta=\delta(\e,n)$ in Proposition \ref{dnslocal-char} which is valid for all $T_i$, and conclude that $(T_i)_\U\in \dns_{n,r}$.  
\end{proof}


\begin{Cor}
\label{dnslocal}
Suppose that $E$ is order continuous with a weak unit, and  let $\U$ be a non-trivial ultrafilter. If $T\in\dns(E,Y)$ then the ultrapower $T_\U\in \dns$. 
\end{Cor} 

As a consequence of the following observation, we shall show that the class of $\dns$ operators is not stable under ultraproducts.

\begin{Rem}\label{dns-q} 
It follows from \cite[Proposition 2.12]{GMM:20} that a closed subspace $M$ of $E$ is dispersed if and only if the quotient map $q:E\to E/M$ is a $\dns$ operator. 
\end{Rem}

\begin{Example}
Let $\U$ be a non-trivial ultrafilter on $\N$. By \cite[Corollary]{LT-II:79}, for each $k\in\N$, there exists a subspace $M_k$ of $L_1\equiv L_1(0,1)$ isometric to $L_{1+1/k}$. 
Since reflexive subspaces of $L_1(0,1)$ are dispersed, if $q_k:L_1\to L_1/M_k$ is the quotient map, then $q_k\in\dns$ for each $k\in \N$ by Remark \ref{dns-q}. 
Let us see that $(q_k)_\U\notin \dns$. 

Since $(q_k)_\U$ acts on $(L_1)_\U$, which is a $L_1(\mu)$ space \cite[Theorem 3.3]{Heinrich}, it is enough to show that $\ker (q_k)_\U$ is not reflexive, and this is true because it is not isomorphic to a subspace of $L_q(\mu)$ for some $q>1$, by the main result of \cite{Rosenthal:73}. 
\end{Example} 

The previous example also shows that the class of dispersed subspaces is not stable under ultraproducts: each $\ker q_k$ is dispersed, but $\ker (q_k)_\U$ is not. However, we can prove the stability for a local variation of the notion of dispersed subspace. 

\begin{Def}\label{n,r-d} 
Suppose that $E$ is order continuous with a weak unit, and let $n\in\N$ and $r>0$. A closed subspace $M$ of $E$ is \emph{$n,r$-dispersed} if for each disjoint set $\{x_1, \ldots, x_n\}$ in $S_E$ there exists $i\in \{1, \ldots, n\}$ so that $\dist(X_i,M)\geq r$.   
\end{Def}

In the conditions of Definition \ref{n,r-d}, a closed subspace $M$ of $E$ is $n,r$-dispersed if and only if the quotient map onto $E/M$ is a $\dns_{n,r}$ operator. Therefore, by Proposition \ref{dnslocal-prod}, $M$ is dispersed if and only if it is $n,r$-dispersed for some $n$ and $r$. 

\begin{Prop} Suppose that $E_i$ is order continuous with a weak unit for each $i\in I$. Let $\U$ be a non-trivial ultrafilter on $I$. If for each $i\in I$, $M_i$ is a ${n,r}$-dispersed subspace then $(M_i)_\U$ is a $n,r$-dispersed subspace of $(E_i)_\U$.
\end{Prop}
\begin{proof}
It is a direct consequence of  Proposition \ref{ultra-dns} and Remark \ref{dns-q}. 
\end{proof}

Observe that $T\in \dns_{n,r}$ and $c>0$ implies $cT\in \dns_{n,cr}$; hence $N(T)$ is not necessarily $n,r$-dispersed, and there is no perturbative characterization for $T\in \dns_{n,r}$. 





\begin{thebibliography}{[99]}

\bibitem{Bilokopytov:21} E. Bilokopytov. \emph{Disjointly non-singular operators on order continuous Banach lattices complement the unbounded norm topology.} J. Math. Anal. Appl. 506 (2022) 125556.

\bibitem{CG:97} J.M.F. Castillo, M. Gonzalez, \emph{Three-Space Problems in Banach Space Theory.} Lecture Notes in Math. 1667, Springer 1997.

\bibitem{DBT:17} Y. Deng, M. O'Brien, V.G. Troitsky. \emph{Unbounded norm convergence in Banach lattices.} Positivity 21 (2017) 963--974. 

\bibitem{DJT:95} J. Diestel, H. Jarchow, A. Tonge. \emph{Absolutely summing operators.} Cambridge Studies in Advanced Mathematics, 43. Cambridge Univ. Press, 1995. 

\bibitem{FHKT:16} J. Flores, F.L. Hern\'andez, N.J. Kalton, P. Tradacete. \emph{Characterizations of strictly singular operators on Banach lattices.} J. London Math. Soc. (2) 79 (2009) 612--630

\bibitem{FLT:16} J. Flores, J. L\'opez-Abad, P. Tradacete. \emph{Banach lattice versions of strict singularity.} J. Funct. Anal. 270 (2016) 2715--2731.

\bibitem{FOPT:22} D. Freeman, T. Oikhberg, B. Pineau, M. A. Taylor. \emph{Stable phase retrieval in function spaces.}
arXiv:2210.05114  


\bibitem{GHSS:00} A. Garc\'\i a Del Amo, F.L. Hern\'andez, V.M. S\'anchez, E.M. Semenov. \emph{Disjointly strictly-singular inclusions between rearrangement invariant spaces.} 
J. London Math. Soc. 62 (2000) 239--252. 

\bibitem{GM:95} M. Gonz\'alez, A. Mart\'\i nez-Abej\'on. \emph{Supertauberian operators and perturbations.} Archiv Math. 64 (1995) 423--433.

\bibitem{GM:97} M. Gonz\'alez, A. Mart\'\i nez-Abej\'on.
\emph{Tauberian operators on $L_1(\mu)$ spaces.} Studia Math. 125 (1997) 289--303.

\bibitem{GM:10} M. Gonz\'alez, A. Mart\'\i nez-Abej\'on. \emph{Tauberian operators.}
Operator Theory: Advances and applications~194.
Birkh\"{a}user, 2010.

\bibitem{GMM:20} M. Gonz\'alez, A. Mart\'\i nez-Abej\'on, A. Martinón. \emph{Disjointly non-singular operators on Banach lattices.} J. Funct. Anal. 280 (2021) 108944, 14 pp.

\bibitem{GM:22} M. Gonz\'alez, A. Martinón. \emph{A quantitative approach to disjointly non-singular operators.} Rev. Real Acad. Ciencias RACSAM (2021) 115:185, 12 pp.

\bibitem{GO:90}  M. Gonz\'alez, V.M. Onieva. \emph{Characterizations of tauberian operators and other semigroups of operators.} Proc. Amer. Math. Soc. 108 (1990) 399--405.

\bibitem{GSK:95} M. Gonz\'alez, E. Saksman, H.-O. Tylli. \emph{Representing non-weakly compact operators.} Studia Math. 113 (1995) 289--303.

\bibitem{Heinrich}
S. Heinrich. \emph{Ultraproducts in Banach space theory.}
J. reine angew. Math. 313 (1980) 72--104.

\bibitem{Hernandez:90} F.L. Hern\'andez. \emph{Disjointly strictly-singular operators in Banach lattices.} Acta Univ. Carolinae -- Math. et Phys. 31 (1990) 35--40.

\bibitem{Hernandez-Salinas:89} F. L. Hern\'andez, B. Rodr\'\i guez-Salinas. \emph{On $\ell_p$ complemented copies in Orlicz spaces II.} Israel J. Math. 68 (1989) 27--55.

\bibitem{JNST:15} W. Johnson, A.B. Nasseri, G. Schechtman, T. Tkocz. \emph{Injective tauberian operators on $L_1$ and operators with dense range in $\ell_\infty$.} Canad. Math. Bull. 58 (2015) 276--280.

\bibitem{KaltonWil:76} N. Kalton, A. Wilansky. \emph{Tauberian operators on Banach spaces.} Proc. Amer. Math. Soc. 57 (1976) 251--255.



\bibitem{LT-II:79} J. Lindenstrauss, L. Tzafriri. \emph{Classical Banach spaces II. Function spaces.} Springer, 1979.


\bibitem{Rosenthal:73} H.P. Rosenthal. \emph{On subspaces of $L^p$.} Ann. Math. 97 (1973) 344-373.

\bibitem{Troitsky:04}  V.G. Troitsky, \emph{Measures of non-compactness on Banach lattices.} Positivity 8 (2004) 165--178.


\bibitem{Wnuk:99} 
W. Wnuk. \emph{Banach Lattices with Order Continuous Norms.} Polish Scientific Publishers PWN, 1999. 

\end{thebibliography}
\end{document}